\newcommand{\si}[1]{#1}
\newcommand{\jo}[1]{}

\si{
\documentclass[10pt,a4paper]{article}
\usepackage[utf8]{inputenc}
\usepackage{graphicx}
\usepackage{amsfonts}
\usepackage{amsthm}
\usepackage{amsmath}
\usepackage{amssymb}
\usepackage{multirow}
\usepackage{array}
\usepackage{hyperref}
\usepackage{xcolor}
\usepackage{multicol}
\usepackage{algorithm}
\usepackage{indentfirst}
\usepackage{algorithmic}
\usepackage{hyperref,cleveref}
\usepackage[left=3cm,right=2cm,top=3cm,bottom=2cm]{geometry}

\tracingstats=0

\newtheorem{theorem}{Theorem}[section]
\newtheorem{proposition}{Proposition}[section]
\newtheorem{corollary}{Corollary}[section]
\newtheorem{lemma}{Lemma}[section]
\newtheorem{definition}{Definition}[section]
\newtheorem{example}{Example}[section]
\newtheorem{remark}{Remark}[section]

\hypersetup{
  colorlinks=true,
  citecolor=blue,
  linkcolor=blue,
  filecolor=magenta,      
  urlcolor=cyan,
}}

\jo{
\RequirePackage{fix-cm}
\documentclass[smallextended]{svjour3}       
\smartqed  

\usepackage[utf8]{inputenc}
\usepackage{graphicx}
\usepackage{amsfonts}
\usepackage{amsmath}
\usepackage{amssymb}
\usepackage{array}
\usepackage{xcolor}
\usepackage{indentfirst}
\usepackage{algorithm}
\usepackage{algorithmic}
\usepackage{chngcntr}
\usepackage{hyperref}

\tracingstats=0
}

\newcommand{\faceq}{\trianglelefteq}			 
\newcommand{\adj}{^*}
\newcommand{\sF}{\mathbb{F}}
\newcommand{\G}{\mathcal{G}}					 

\newcommand{\C}{\mathcal{C}}					 
\newcommand{\E}{\mathbb{E}}					 
\renewcommand{\S}{\mathbb{S}}					 

\renewcommand{\bar}{\overline}                   
\newcommand{\I}{\mathbb{I}}                      
\newcommand{\K}{\mathcal{K}}                     
\newcommand{\lin}{\mathrm{lin}}                  
\newcommand{\Sp}{\mathrm{span}}                  
\newcommand{\Ker}{\mathrm{Ker\hspace{0.03cm}}}                 
\renewcommand{\Im}{\mathrm{Im\hspace{0.03cm}}}                 
\renewcommand{\int}{\mathrm{int\hspace{0.03cm}}}               
\newcommand{\rank}{\mathrm{rank }}               
\newcommand{\F}{\Omega}                     
\newcommand{\R}{\mathbb{R}}  					 
\newcommand{\T}{\top\hspace{-1pt}}               
\newcommand{\pol}{^{\textnormal{o}}}             
\newcommand{\xb}{\bar{x}}                        
\newcommand{\yb}{\bar{y}}                        

\newcommand{\spn}{\textnormal{span}}
\newcommand{\cl}{\textnormal{cl}}
\newcommand{\Fmin}{F_{\min}}
 
\newcommand{\ri}{\textnormal{ri}}				

\makeatletter
\newcommand{\tpitchfork}{%
  \vbox{
    \baselineskip\z@skip
    \lineskip-.52ex
    \lineskiplimit\maxdimen
    \m@th
    \ialign{##\crcr\hidewidth\smash{$-$}\hidewidth\crcr$\pitchfork$\crcr}
  }%
}
\makeatother                                     

\jo{
\usepackage[misc]{ifsym}
\newcommand{\corr}{${\textnormal{\Letter}}$}
}

\begin{document}

\title{A minimal face constant rank constraint qualification for reducible conic programming
\footnotetext{The authors received financial support from
 CEPID - FAPESP (grant 
 2013/07375-0),  
FAPESP (grants 
2018/24293-0, 
2017/18308-2, and 
2017/17840-2), 
CNPq (grants 
301888/2017-5, 
303427/2018-3, 
404656/2018-8, and
306988/2021-6), 
PRONEX - CNPq/FAPERJ (grant 
E-26/010.001247/2016), and 
FONDECYT grant 
1201982 and 
Centro de Modelamiento Matem{\'a}tico (CMM), ACE210010 and FB210005, BASAL funds for center of excellence, all from ANID-Chile.}}

\jo{\titlerunning{A minimal face constant rank CQ for reducible conic programming}}

\jo{\journalname{Mathematical Programming}}

\si{
\author{
Roberto Andreani \thanks{Department of Applied Mathematics, State University of Campinas, Campinas, SP, Brazil. 
	Email: {\tt andreani@unicamp.br}}
\and
Gabriel Haeser \thanks{Department of Applied Mathematics, University of S{\~a}o Paulo, S{\~a}o Paulo, SP, Brazil. 
	Emails: {\tt ghaeser@ime.usp.br, leokoto@ime.usp.br}}	
\and 
Leonardo M. Mito \footnotemark[2]
\and
H{\'e}ctor Ram{\'i}rez \thanks{Departamento de Ingenier{\'i}a Matem{\'a}tica and Centro de Modelamiento Matem{\'a}tico (AFB170001 - CNRS UMI 2807), Universidad de Chile, Santiago, Chile.
Email: {\tt hramirez@dim.uchile.cl}}
}
}

\jo{
\author{
Roberto Andreani
\and
Gabriel Haeser
\and 
Leonardo M. Mito
\and
H{\'e}ctor Ram{\'i}rez
}

\authorrunning{R. Andreani, G. Haeser, L. M. Mito, and H. Ram{\'i}rez} 

\institute{ 
    Roberto Andreani \at
    Department of Applied Mathematics, 
    University of Campinas, 
    S{\~ao} Paulo, Brazil. \\
	\email{andreani@unicamp.br}
    \and
    Gabriel Haeser \corr \at
    Department of Applied Mathematics, 
    University of S\~ao Paulo, 
    S{\~a}o Paulo, Brazil.\\
    \email{ghaeser@ime.usp.br} 
    \and
    Leonardo M. Mito  
    \at
    Department of Applied Mathematics, 
    University of S{\~a}o Paulo, 
    S{\~a}o Paulo, Brazil.\\
	\email{leonardommito@gmail.com}
	\and	
	H{\'e}ctor Ram{\'i}rez \at
	Department of Mathematical Engineering and Center for Mathematical Modeling (CNRS IRL 2807),
	University of Chile, 
	Santiago, Chile.\\
	\email{hramirez@dim.uchile.cl}
  }
}

\jo{\date{Received: date / Accepted: date}}
\si{\date{April 26, 2023. Last revised on November 30, 2024.}}

\maketitle

\begin{abstract}
In a previous paper {\it [Andreani et al, Math. Prog. 202, p. 473--514, 2023]} we introduced a constant rank constraint qualification for nonlinear semidefinite and second-order cone programming by considering all faces of the underlying cone. This condition is independent of Robinson's condition and it implies a strong second-order necessary optimality condition which depends on a single Lagrange multiplier instead of the full set of Lagrange multipliers. In this paper we expand on this result in several directions, namely, we consider the larger class of $\mathcal{C}^2-$cone reducible constraints and we show that it is not necessary to consider all faces of the cone; instead a single specific face should be considered (which turns out to be weaker than Robinson's condition) in order for the first order necessary optimality condition to hold. This gives rise to a notion of facial reduction for nonlinear conic programming, that allows locally redefining the original problem only in terms of this specific face instead of the whole cone, providing a more robust formulation of the problem in which Robinson's condition holds. We were also able to prove the strong second-order necessary optimality condition in this context by considering  only the subfaces of this particular face, which is a new result even in nonlinear programming.
\end{abstract}

\si{
\

}

\jo{
}

\section{Introduction}

In this paper we are interested in the general nonlinear conic optimization problem with smooth data. In most references on this problem, at least Robinson's constraint qualification is usually assumed in order to guarantee certain stability of the underlying problem. Under this condition, the set of Lagrange multipliers at a local solution is non-empty and bounded and a second-order necessary optimality condition may also be formulated. On the other hand, for standard nonlinear programming problems where the cone is the non-negative orthant, several other constraint qualifications are available with different characteristics. Recently a geometric definition of a constant rank constraint qualification  has been proposed in \cite{CRCQfaces}, extending the well known nonlinear programming concept proposed by Janin in \cite{crcq} to the context of nonlinear second-order cone and semidefinite programming. In this paper we propose a much more general approach for general reducible cones, where the proofs turn out to be significantly simplified while considerably broadening the results to a class of problems that include, for instance, optimization problems over the cone defined by the epigraph of the Ky Fan matrix $k$-norm~\cite{ding-kyfan} as well as cartesian products and certain preimages of other reducible cones~\cite[Propositions 3.1 and 3.2]{shapiro-sensitivity}.

A key concept in our analysis is cone-reducibility, which is able to treat every point of a cone as a point on the vertex of a reduced cone. Then, the constant rank conditions formulated for second-order cone and semidefinite programming may be stated as the local constant dimension of a linearization of the feasible set along the subspace orthogonal to each of the faces of the underlying cone, which turns out to be the adequate definition for any general reducible cone.  Besides showing that this condition is able to ensure existence of Lagrange multipliers (which may now be an unbounded set), we also prove a second-order necessary optimality condition that is stronger than the one fulfilled at local solutions that satisfy Robinson's condition.
Namely, we show that under our constant rank condition (which is independent of Robinson's condition), the usual second-order term associated with a nonlinear conic programming (that is, the Hessian of the Lagrangian minus the so-called ``sigma term'', evaluated at any direction of the critical cone) is non-negative and invariant to the choice of the Lagrange multiplier. This shows that {\it any} Lagrange multiplier can be used to check the standard second-order necessary optimality condition instead of considering the whole set of Lagrange multipliers to deal with different directions on the critical cone, as it is the case when one assumes Robinson's condition to hold.

Somewhat surprisingly, differently from what was done in \cite{CRCQfaces}, we were able to formulate two different conditions, both weaker than the constant rank condition; one for the first-order results and a slightly stronger one for the second-order results. Namely, instead of considering {\it all} faces of the cone in the formulation of our constant rank assumption, we ask for the constant rank along one particular face $F$, which is the minimal face of the cone that contains the standard linearization cone. The condition for the second-order result is obtained considering all subfaces of $F$. The first-order condition is a generalization of what is known as {\it constant rank of the subspace component} (CRSC, \cite{cpg}) in nonlinear programming (strictly weaker than Robinson's/Mangasarian-Fromovitz's condition), while the second-order result is new even in this context. Notice that it is well known that the strong second-order necessary optimality condition does not hold under Robinson's condition or CRSC (see the extended version of \cite{conjnino}); thus our stronger version of CRSC (which is independent of Robinson's condition) provides an adequate new condition, even for nonlinear programming, which gurantees the validity of the strong second-order necessary optimality condition.

Acknowledging that Robinson's condition is a very desirable property of a conic programming problem to have, we provide a useful application of CRSC by showing that it allows a conic programming problem to be locally equivalently rewritten by replacing the original cone by the minimal face $F$, where the reformulated problem satisfies Robinson's condition. This procedure is well known in the literature of linear conic programming \cite{wolkowicz1981,wolkowicz1997,wolkowicz2012,wolkowicz2013,waki2013,pataki2013,tuncel} but its extension to the nonlinear case has not been previously considered. Unaware of the developments in the conic programming literature, in \cite{cpg} it was shown that for a nonlinear programming problem that satisfies CRSC, a particular index set of inequality constraints are locally equivalent to equality constraints, while \cite{shulu} shows that the reformulated problem satisfies Robinson's condition under a constant rank assumption. This can be viewed as a nonlinear programming analogue of the facial reduction procedure known in linear conic programming. Our result extends both approaches by providing a nonlinear facial reduction procedure in the general context of nonlinear conic programming.\\

{\bf Summary of contributions:}
\begin{itemize}
\item We consider the constant rank constraint qualification (CRCQ) defined in \cite{crcq} for nonlinear programming and in \cite{CRCQfaces} for nonlinear second-order cone and semidefinite programming and we extend its definition to the more general class of nonlinear $\mathcal{C}^2-$cone reducible constraints, showing the existence of a Lagrange multiplier at local solutions;
\item Under CRCQ a strong second-order necessary optimality condition is proved to hold; the difference with the standard second-order condition (obtained under Robinson's condition) is that it can be checked with {\it any} Lagrange multiplier, obtaining thus a stronger second-order necessary optimality condition, even though CRCQ is a condition independent of Robinson's condition.
\item An explanation of the previous result is presented by means of showing that CRCQ guarantees that the standard second-order term evaluated at a given direction of the critical cone is independent of the chosen Lagrange multiplier. This result was previously proved only for nonlinear programming \cite{gfrerer,conjnino}; however, even in this context, this result does not seem to be well-known.
\item All previous results were in fact obtained for general reducible cones under two statements weaker than CRCQ, which considerably simplifies the proofs. That is, by considering only a particular face $F$ of the underlying reduced cone (instead of all of them), we define the constant rank of the subspace component (CRSC, \cite{cpg}), which is weaker than Robinson's condition but is enough for proving most results presented in the paper. The exception corresponds to our second-order result, which requires a stronger variant of CRSC (strong-CRSC) formulated in terms of all subfaces of $F$. The latter is new even in nonlinear programming.
\item Finally, under CRSC, we show that a nonlinear conic programming problem can be recast by considering in the constraint only the particular face $F$ instead of the whole cone, which provides a locally equivalent problem where Robinson's condition is satisfied. This generalizes the identification of inequality constraints that locally behave as equalities, known in nonlinear programming, while also extending what is known as the facial reduction procedure in linear conic programming, showing that these apparent non-related results are in fact two manifestations of the same phenomenon.
\end{itemize}

{\bf Notation:}

For any positive integer $n$, we denote by $\R^n$ the set of $n$-dimensional vectors of real numbers equipped with the usual inner product and corresponding norm $\|\cdot\|$. By $\R^n_+$ we denote the subset of $\R^n$ composed of those vectors with non-negative entries. For $G:\mathbb{F}\to\E$, where $\mathbb{F}$ and $\E$ are finite dimensional vector spaces with inner product $\langle\cdot,\cdot\rangle$, we denote by $DG(\xb)[\cdot]:\mathbb{F}\to\E$ the derivative of $G$ at $\xb\in\mathbb{F}$ and $DG(\xb)^*[\cdot]$ its adjoint operator. The second-order derivative of $G$ at $\xb\in\mathbb{F}$ evaluated at $d\in\mathbb{F}$ is denoted by $D^2G(\xb)[d,d]\in\E$. Given a set $C\subseteq\E$, the polar cone of $C$ is denoted by $C\pol\doteq\{z\in\E:\langle z,x\rangle\leq0,\forall x\in C\}$. The interior and relative interior of $C$ are denoted respectively by $\int(C)$ and $\ri(C)$, while the subspace generated by $C$ is denoted by $\Sp(C)$  and its lineality space is denoted by $\lin(C)$.
We denote by $C^\perp$ the subspace orthogonal to $\Sp(C)$. The image and kernel of a linear operator $A\colon\mathbb{F}\to\E$ are denoted respectively by $\Im A$ and $\Ker A$. The cardinality of a finite set $S$ is denoted by $|S|$.

\section{Constraint qualifications, reducibility and faces}\label{background}

Let $\E$ be a finite-dimensional linear space equipped with an inner product $\langle \cdot , \cdot \rangle$, and let $\K\subseteq \E$ be a non-empty closed convex cone. The main object of our study in this paper is the following problem:
\begin{equation}
	\begin{array}{ll}
		\underset{x\in \R^n}{\nonumber\mbox{\textnormal{Minimize }}} 	& f(x),\\ 
		\label{NCP}\mbox{subject to}	& G(x)\in \K,
	\end{array}
	\tag{NCP}
\end{equation}
where $f\colon \R^n\to \R$ and $G\colon \R^n\to \E$ are at least twice continuously differentiable functions. We will denote the feasible set of \eqref{NCP} by $\Omega\doteq G^{-1}(\K)$. This general formulation covers important classes of optimization problems such as nonlinear programming (NLP) with $\K=\R^m_+$; nonlinear semidefinite programming (NSDP) with $\K=\S^m_+\doteq\{ A\in \R^{m\times m} : v^\top A v \geq 0 \, \forall v\in \R^m\}$, and nonlinear second-order cone programming (NSOCP) with 
$\K= \Pi_{i=1}^M \mathcal{L}_{m_i}$ and $\mathcal{L}_{m_i}\doteq\{ v\in \R^{m_i} : v_1  \geq \| (v_2,...,v_{m_i})\|\}$. For simplicity we do not consider equality constraints, however, equality constraints may be included in a somewhat standard way similarly to the discussion conducted in \cite{CRCQfaces}.

To contextualize our contribution and introduce some notation, let us recall some general aspects of~\eqref{NCP}. To start let $\xb\in \Omega$ and consider the set
\[
	T_{\Omega}(\xb)\doteq \{d\in \R^n\colon \textnormal{dist}(\xb+td,\Omega)=o(t), \ \forall t>0\}
\] which is known as the (Bouligand) \textit{tangent cone} to $\Omega$ at $\xb$. For nonlinear optimization problems, the set $T_{\Omega}(\xb)$ is not always directly computable, but it admits an outer approximation in terms of the tangent cone to $\K$ at $G(\xb)$, defined as follows:
\[
	L_{\Omega}(\xb)\doteq DG(\xb)^{-1}(T_{\K}(G(\xb)))= \{d\in \R^n\colon DG(\xb)d\in T_{\K}(G(\xb))\}.
\]
This cone is often called the \textit{linearized tangent cone} (or simply \textit{linearized/linearization cone}) to $\Omega$ at $\xb$, and as mentioned before, $T_{\Omega}(\xb)\subseteq L_{\Omega}(\xb)$. Moreover, as it was shown by Guignard~\cite{guignard} and others, the polar of $L_\Omega(\xb)$ is given by the closure of the set
\begin{equation}\label{def:H}
H(\xb)\doteq DG(\xb)^*[T_\K(G(\xb))\pol]\doteq\{DG(\xb)^*[Y]\colon Y\in T_\K(G(\xb))\pol\}
\end{equation}
and hence if $H(\xb)$ is closed, then $L_{\Omega}(\xb)\pol=H(\xb)$ -- see~\cite[Proposition 1]{guignard} or alternatively~\cite[Lemma 2.1]{CRCQfaces}. Characterizing the polar cone of $L_{\Omega}(\xb)$ is of great importance in the study of optimality conditions for~\eqref{NCP} because one of the classical ways of doing so is by noticing that if $\xb$ is a local minimizer of~\eqref{NCP}, then $-\nabla f(\xb)\in T_{\Omega}(\xb)\pol$, and if, in addition, the equality $T_{\Omega}(\xb)\pol=L_{\Omega}(\xb)\pol$ holds true and $H(\xb)$ is closed, then 
\begin{equation}\label{eq:geokkt}
	-\nabla f(\xb)\in H(\xb).
\end{equation}
Equation~\eqref{eq:geokkt} is widely known in the literature as the \textit{Karush/Kuhn-Tucker} (KKT) condition, and every element of the set
\[
	\Lambda(\xb)\doteq \{\bar{Y}\in T_{\K}(G(\xb))\pol\colon -\nabla f(\xb)=DG(\xb)^*[\bar{Y}]\}
\]
is called a \textit{Lagrange multiplier} associated with $\xb$. Furthermore, any condition (not depending on the objective function $f$) that ensures that the KKT condition is necessary for local optimality of a given $\xb\in \Omega$ is called a \textit{constraint qualification} (CQ) for~\eqref{NCP} at $\xb$. For instance, the condition we just mentioned:
\begin{center}
	$T_{\Omega}(\xb)\pol=L_{\Omega}(\xb)\pol$ and $H(\xb)$ is closed
\end{center}
is known as \textit{Guignard's CQ}~\cite{guignard}.

Stronger constraint qualifications may have additional potentially interesting properties; for instance, the classical \textit{Robinson's CQ}~\cite{Rob76} (at $\xb$):
\begin{equation}\label{def:robinson}
	0\in \int(\Im DG(\xb)-\K+G(\xb))
\end{equation}
ensures that $\Lambda(\xb)$ is a nonempty compact set~\cite[Theorem 3.9]{bonnans-shapiro}, which is a powerful tool for obtaining results regarding convergence of algorithms and error bounds, whereas the \textit{transversality} (or \textit{nondegeneracy}) condition (at $\xb$):
\begin{equation}\label{def:ndg}
\Im DG(\xb)+\lin(T_{\K}(G(\xb)))=\E,
\end{equation}
originally presented for nonlinear semidefinite programming by Shapiro and Fan~\cite{shapiro1995eigenvalue} and later generalized to~\eqref{NCP} by Shapiro~\cite{shapiro-uniqueness}, implies that $\Lambda(\xb)$ is a singleton; for this reason, nondegeneracy is also widely used as a standard CQ in the conic programming literature.

Recently, in Andreani et al.~\cite{CRCQfaces}, it was introduced a constant rank-type CQ for nonlinear semidefinite programming and nonlinear second-order cone programming that, as we interpret, stands upon two basic geometric notions from the literature: \textit{cone reducibility} and \textit{faces}. The first contribution of this paper is to extend this condition (along with all of its known properties) to a general context of~\eqref{NCP} where the notion of reducibility may still be applied, and then we will proceed to our main results. To do so, let us first review these fundamental concepts, starting with cone reducibility, following the classical book of Bonnans and Shapiro~\cite{bonnans-shapiro}.

Roughly speaking, the notion of reducibility allows one to visualize any given point of $\K$ as the vertex of a potentially ``smaller'' pointed cone $\C$ in a smooth way. In more accurate terms:

\begin{definition}[Definition 3.135 of~\cite{bonnans-shapiro}]
Let $\E$ and $\mathbb{F}$ be finite-dimensional linear spaces, and let $\K\subseteq \E$ and $\C\subseteq \mathbb{F}$ be non-empty closed convex cones; moreover, assume that $\C$ is pointed. We say that $\K$ is \textit{reducible} to $\mathcal{C}$ at a point $Y\in \K$ if there exists a neighborhood $\mathcal{N}$ of $Y$ and a twice continuously differentiable\footnote{The twice continuous differentiability of $\Xi$ ($C^2$-reducibility) is only required for the results presented in Section 4. For the remaining parts of the paper it is enough to assume continuous differentiability.} reduction function $\Xi\colon \mathcal{N}\to \mathbb{F}$ (possibly depending on $Y$) such that 
\begin{enumerate}
\item $\Xi(Y)=0$;
\item $D\Xi(Y)$ is surjective;
\item $\K\cap \mathcal{N}=\{Z\in \mathcal{N}\colon \Xi(Z)\in \mathcal{C}\}.$
\end{enumerate}
\end{definition}

For the particular case of NLP, where $\E=\R^m$, $\K=\R^m_+$, and $G(x)\doteq (g_1(x),\ldots,g_m(x))$, an example of reduction is to isolate \textit{active constraints} -- that is, the constraints indexed by $\mathcal{A}(\xb)\doteq \{j\in \{1,\ldots,m\}\colon g_j(\xb)=0\}$ -- in a neighborhood of $G(\xb)$, for a given feasible point $\xb\in \Omega$. Indeed, in this case we have $\C=\R^{|\mathcal{A}(\xb)|}_+$ and the reduction mapping is given by $\Xi(y_1,\ldots,y_m)\doteq (y_i)_{i\in \mathcal{A}(\xb)}$. Then, in a neighborhood of $\xb$, the original constraint $G(x)\in \R^m_+$ is equivalent to the reduced constraint 
\begin{equation}\label{eq:nlpred}
	\G(x)\in \R^{|\mathcal{A}(\xb)|}_+,\mbox{ where }\G(x)\doteq \Xi(G(x))=(g_i(x))_{i\in \mathcal{A}(\xb)}.
\end{equation}

On the other hand, it is well-known that second-order and semidefinite cones are reducible at every point (see \cite{BonRam} and \cite[Example 3.140]{bonnans-shapiro}, respectively). Indeed, for the case of a second order cone  $\mathcal{L}_{m}$ (for some given $m$), a reduction mapping depending on the point $\yb \in \mathcal{L}_{m}$ may be defined as follows: when $\yb\in \int ( \mathcal{L}_{m})$, then we can trivially choose its reduction as identically zero and $\C=\{0\}$; when $\yb=0$ then its reduction is the identity function and $\C=\mathcal{L}_{m}$; while when $\yb$ is in the non-zero boundary of $\mathcal{L}_{m}$, then $\Xi(y)\doteq y_1 -\| (y_2,\dots,y_m)\|$ and $\C=\R_+$. For the case of positive semidefinite matrices, for a given $\bar A\in \S^m_+$ we can locally construct an analytic function $E(\cdot)$ such that the columns of $E(\bar A)$ are an orthornormal eigenbasis of $\Ker \, \bar A$ and the columns of $E(A)$ correspond to an orthornormal basis of the eigenspace associated with the smallest $m-r$ eigenvalues of $A$, being $r\doteq \rank\, \bar A$. Then, a reduction for $\S^m_+$ at $\bar A$ is given by $\C\doteq\S^{m-r}_+$ and
\begin{equation}\label{eq:red_SDP}
	\Xi(A)\doteq E(A)^\top A E(A).
\end{equation}

Hence, the context of general cone-constrained optimization problems is quite similar to the NLP case: if $\K$ is reducible to $\C$ at $G(\xb)$ by the reduction function $\Xi$, define $\G\doteq \Xi\circ G$ and the reduced constraint $\G(x)\in \C$ turns out to be equivalent to the original constraint $G(x)\in \K$ in a sufficiently small neighborhood of $\xb$ with $\G(\xb)=0$. By replacing $\mathbb{F}$ with the affine hull of $\C$, we will assume that $\C$ has nonempty interior in $\mathbb{F}$. With this in mind, from this point onwards we will focus on the reduced problem
\begin{equation}
	\begin{array}{ll}
		\underset{x\in \R^n}{\nonumber\mbox{\textnormal{Minimize }}} 	& f(x),\\ 
		\label{redNCP}\mbox{subject to}	& \G(x)\in \C
	\end{array}
	\tag{Red-NCP}
\end{equation}
around a given point $\xb\in\F$ which will be clear from the context.
In order to translate results about~\eqref{redNCP} back to the original problem language, we will briefly recall some classical correspondences between~\eqref{NCP} and~\eqref{redNCP} that can also be found, for instance, in Bonnans and Shapiro's book~\cite[Section 3.4.4]{bonnans-shapiro}. 

As a start, observe that
\[
	D\Xi(G(\xb))[T_\K(G(\xb))]=T_\C(\G(\xb))= \C
\]
because $\G(\xb)=0$ and, therefore, $T_\K(G(\xb))\pol=D\Xi(G(\xb))^*[\C\pol]$, as it is also stated in~\cite[Equation 3.267]{bonnans-shapiro}. Then, we use the chain rule
\[
	D\G(\xb)[ \ \cdot \ ]=D\Xi(G(\xb))[DG(\xb)[ \ \cdot \ ]]
\]
to conclude that
\begin{equation}\label{eq:Lred-L}
	\begin{aligned}
		L_\Omega(\xb) & =\bigcap_{\mu\in T_\K(G(\xb))\pol}\{d\in \R^n\colon \langle DG(\xb)[d],\mu \rangle\leqslant 0\}\\
		& =\bigcap_{\eta\in\C\pol}\{d\in \R^n\colon \langle DG(\xb)[d],D\Xi(G(\xb))^*[\eta] \rangle\leqslant 0\}\\
		& =\bigcap_{\eta\in\C\pol}\{d\in \R^n\colon \langle D\G(\xb)[d],\eta \rangle\leqslant 0\}\\
		& = \{d\in \R^n\colon D\G(\xb)[d]\in \C\};
	\end{aligned}
\end{equation}
that is, the linearized cone of~\eqref{NCP} coincides with the linearized cone of the reduced problem~\eqref{redNCP}. Similarly, recall the set $H(\xb)$ defined in~\eqref{def:H} and notice that
\begin{equation}\label{eq:hred-h}
	\begin{aligned}
	H(\xb) & = DG(\xb)^*[T_{\K}(G(\xb))\pol] \\
	& = DG(\xb)^*[D\Xi(G(\xb))^*[\C\pol]]\\
	& = D\G(\xb)^*[\C\pol],
	\end{aligned}
\end{equation}
meaning it also coincides with its counterpart defined for the reduced problem~\eqref{redNCP}. Moreover, the Lagrange multiplier sets $\Lambda(\xb)$ and $\mathcal{M}(\xb)$, of~\eqref{NCP} and~\eqref{redNCP} respectively, both associated with $\xb$, satisfy the relation
\[
	\Lambda(\xb)=D\Xi(G(\xb))^*[\mathcal{M}(\xb)].
\]

This shows that the results we obtain for the reducible problem \eqref{redNCP} may be translated to \eqref{NCP} in a standard way. Now, a key element in our analysis is the faces of $\C$, which we recall the definition as follows:

\begin{definition}[Faces of a convex set]
Let $\C\subseteq\E$ be a closed convex set. We say that $F$ is a \emph{face} of $\mathcal{C}$ when $F$ is a convex subset of $\C$ such that for every $y\in F$ and every $z,w\in \mathcal{C}$ such that $y=\alpha z + (1-\alpha)w$ for some $\alpha\in (0,1)$, we have that $z,w\in F$.  We use the notation $F\faceq \mathcal{C}$ to say that $F$ is a face of $\mathcal{C}$. 
\end{definition}

We recall also that given a set $C\subseteq \C$, the \emph{minimal face} associated with $C$, denoted by $\Fmin(C)$, is defined as the smallest face of $\C$ that contains $C$. We also recall two very useful properties about minimal faces from Pataki~\cite{pataki-geometry}:

\begin{lemma}[Proposition 3.2.2 of~\cite{pataki-geometry}]\label{lem:pataki}
Let $\C\subseteq\E$ be a closed convex cone, $F$ be a face of $\C$, and $C\subseteq \C$ be a convex set. Then:
\begin{enumerate}
\item $F=\Fmin(C)$ if, and only if, $\ri(F)\cap \ri(C)\neq \emptyset$;
\item If $C\subseteq F$ and $C\cap \ri(F)\neq \emptyset$, then $F=\Fmin(C)$.
\end{enumerate}
\end{lemma}
In particular, $\Fmin(\{c\})$ is the unique face of $\C$ that contains $c$ in its relative interior. 
%
%

%

\section{Extending and improving the constant rank condition}

Our extended definition of the constant rank constraint qualification proposed in~\cite{CRCQfaces} (originally stated for NSDP and NSOCP) can be formulated as follows for the general problem~\eqref{NCP}, for any given point $\xb\in \Omega$ such that $\K$ is locally reducible to a cone $\C\subseteq \mathbb{F}$ in a neighborhood of $G(\xb)$ by some reduction mapping $\Xi$:

\begin{definition}[CRCQ]\label{def:crcq} The \emph{facial constant rank} (FCR) property holds at $\xb$ if there exists a neighborhood $\mathcal{V}$ of $\xb$ such that, for every $F\faceq\C$, the dimension of $D\G(x)^*[F^\perp]$ remains constant for every $x\in \mathcal{V}$. Furthermore, the \emph{constant rank constraint qualification} (CRCQ) holds at $\xb$ if it satisfies the facial constant rank property and, additionally, the set $H(\xb)$ defined in~\eqref{def:H} is closed. 
\end{definition}

Observe that the constant dimension required in Definition~\ref{def:crcq} need not be the same for different faces, but it must remain the same for different $x\in \mathcal{V}$, given any face of $\C$. 

Proving that CRCQ as in Definition~\ref{def:crcq} is indeed a CQ for~\eqref{NCP}, on the other hand, is not trivial from the proofs presented in~\cite[Theorems 4.1 and 5.1]{CRCQfaces}. But instead of proving directly that this extension of CRCQ is a constraint qualification, we shall first present and discuss a refined form of it and then, at the end of this section, we prove that this new refined condition is a constraint qualification which is a conclusion that naturally extends to Definition~\ref{def:crcq}. The core of such refinement lies in the fact that not all faces of $\C$ are necessarily needed; instead, only one face is required.

To get some intuition on which face is the good one, let us drive our attention to the particular case of NLP. That is, choose $\E=\R^m$ and $\K=\R^m_+$, and in this case problem~\eqref{NCP} takes the form:
\begin{equation}
	\begin{array}{ll}
		\nonumber\mbox{\textnormal{Minimize }} 	& f(x),\\ 
		\label{NLP}\mbox{s.t. }	& G(x)\doteq(g_1(x),\ldots,g_m(x))\in \R^m_+,
	\end{array}
	\tag{NLP}
\end{equation}
with the reduction mapping given by~\eqref{eq:nlpred} at some $\xb\in \Omega$. For the sake of simplicity, we will assume that $\mathcal{A}(\xb)=\{1,\ldots,m\}$ in the next paragraphs, so $\G=G$ and $\C=\K=\R^m_+$. In this case, CRCQ as in Definition~\ref{def:crcq} becomes equivalent to the constant rank of the family $\{\nabla g_i(x)\}_{i\in J}$ for $x$ in a neighborhood of $\xb$, for each $J\subseteq \{1,\ldots,m\}$, which follows from the fact that $F\faceq \R^m_+$ if, and only if,
\[
	F=\R^m_+\bigcap_{j\in J} \{e_j\}^\perp
\]
for some $J\subseteq \{1,\ldots,m\}$ and, moreover, $F$ and $J$ are in a one-to-one correspondence. Above, $e_j$ denotes the $j$-th vector of the canonical basis of $\R^m$. The requirement of taking every subset $J$ into consideration was relaxed in~\cite{cpg}, where the authors realized that only
\[
	J_-\doteq \{j\in \{1,\ldots,m\}\colon -\nabla g_j(\xb)\in L_\Omega(\xb)\pol\}
\]
was needed. This gave rise to a condition they called the \textit{constant rank of the subspace component} (CRSC)~\cite[Definition 1.3]{cpg}, which holds at $\xb$ when the family of vectors $\left\{ \nabla g_j(x) \right\}_{j\in J_-}$ remains with constant rank for every $x$ in a neighborhood of $\xb$. The CRSC condition was then studied by Kruger et al.~\cite{positivity} (under the name {\it relaxed Mangasarian-Fromovitz condition}), where they noticed that the set $J_-$ can be equivalently written as
\begin{equation}\label{eq:J-}
	J_-=\{j\in \{1,\ldots,m\}\colon \langle \nabla g_j(\xb), d \rangle=0, \ \forall d\in L_{\Omega}(\xb)\}.
\end{equation}
Then, for every $d\in L_{\Omega}(\xb)$ and every $j\in J_-$, we obtain from \eqref{eq:J-} that
\[
	\langle DG(\xb) d, e_j\rangle=\langle\nabla g_j(\xb),d\rangle=0,
\]
hence
\[
	DG(\xb)[L_{\Omega}(\xb)]\subseteq \R^m_+\bigcap_{j\in J_-} \{e_j\}^\perp\doteq F_{J_-},
\]
where $F_{J_-}$ is the face of $\R^m_+$ defined by the set $J_-$. Notice that for each $j\notin J_-$ there is some $d_j\in L_{\Omega}(\xb)$ such that $\langle DG(\xb) d_j, e_j\rangle>0$ so $d_{J_-}\doteq \sum_{j\in J_-} d_j\in L_{\Omega}(\xb)$ satisfies  $\langle DG(\xb) d_{J_-}, e_j\rangle>0$ for every $j\notin J_-$ and, consequently, 
\[
	DG(\xb)d_{J_-}\in \ri(F_{J_-})\cap DG(\xb)[L_{\Omega}(\xb)]. 
\]
Thus, we conclude from Lemma \ref{lem:pataki} that
\[
	F_{J_-}=\Fmin\left(DG(\xb) [L_{\Omega}(\xb)]\right).
\] 

Back to the general problem~\eqref{NCP} and its reduced counterpart \eqref{redNCP}, we consider the subset $\Gamma_\C(\xb)$ of $\C$ given by
\begin{equation}\label{def:gamma}
	\Gamma_{\C}(\xb)\doteq D\G(\xb)[L_{\Omega}(\xb)]
\end{equation}
and, based on the previous discussion, we consider the face $F_{J_-}\doteq \Fmin(\Gamma_{\C}(\xb))\faceq \C$. Then, we can define an extension of the CRSC condition to the context of NCP as follows:

\begin{definition}[CRSC]\label{def:crsc}
We say that the \emph{constant rank of the subspace component} (CRSC) condition for \eqref{NCP} holds at $\xb\in\F$ if $H(\xb)$ is closed and there exists a neighborhood $\mathcal{V}$ of $\xb$ such that the dimension of $D\G(x)^*[F_{J_-}^\perp]$ remains constant for every $x\in \mathcal{V}$, where $F_{J_-}\doteq \Fmin(\Gamma_{\C}(\xb))\faceq \C$.
\end{definition}

From our previous discussion, it follows immediately that Definition~\ref{def:crsc} fully recovers the CRSC condition from NLP when it is seen as a particular case of~\eqref{NCP}, since $H(\xb)$ is always closed in this case. Moreover, CRSC as in Definition~\ref{def:crsc} is clearly implied by CRCQ but the following example shows that this implication is strict (that is, CRSC does not imply CRCQ).


\begin{example}\label{ex:crscnotcrcq}
Let $\K\doteq \S^2_+$ and $\bar{x}\doteq (0,0)\in \R^2$ and consider the NSDP constraint:
\[
	G(x)\doteq \begin{bmatrix}
	x_1 & x_1 x_2 \\
	x_1 x_2 & x_2^2+x_2
	\end{bmatrix}\in \S^2_+
\]
which is reducible at $\xb$ via the identity function with $\C=\K$ and $\G=G$, since $G(\bar{x})=0$ is already at the vertex of $\S^2_+$. With respect to this choice of reduction, we will show that CRSC holds true at $\bar{x}$. To do so, observe that the linear functional $D\G(\xb)$ satisfies the following:
\[
	D\G(\xb)d = \begin{bmatrix} d_1 & 0 \\ 0 & d_2\end{bmatrix}
\]
for each $d\doteq(d_1,d_2)\in\R^2$. It follows that $L_{\Omega}(\xb)=\left\{d\in \R^2\colon D\G(\bar{x})d\in\S^2_+\right\} = \R^2_+$
and
\[
	\Gamma_{\C}(\xb) = D\G(\xb)[L_\Omega(\bar{x})] = \left\{
		\begin{bmatrix} d_1 & 0 \\ 0 & d_2\end{bmatrix}\in \S^2 \colon d_1\geq 0, d_2\geq 0	
	\right\}
\]
so $F_{J_-}=F_{\min}(\Gamma_{\C}(\xb))=\S^2_+$. Therefore, for every $x\in\R^2$, we have $D\G(x)^*[F_{J_-}^\perp] = \{0\}$, which has constant dimension. Now it suffices to note that $H(\bar{x})=DG(\xb)^*[\K\pol] = \{(z_1, z_2)\in\R^2\colon z_1\leq 0, z_2\leq 0\}$ is closed. On the other hand, let
\[
	F\doteq \left\{
		\begin{bmatrix} z & 0 \\ 0 & 0\end{bmatrix}\in \S^2 \colon  z\geq 0\right\}\faceq \S^2_+
\]
and we have that 
\[
D\G(x)^*[F^\perp] = \{(\alpha x_2, \ \alpha x_1+\beta(2x_2+1))\in \R^2\colon \alpha,\beta\in\R \}
\]
so $\dim(D\G(\xb)^*[F^\perp])=1$ but $\dim(D\G(x)^*[F^\perp])=2$ whenever $x_2\neq 0$. Thus, CRCQ does not hold at $\xb$.
\end{example}
 
Moreover, observe that $\xb$ from Example~\ref{ex:crscnotcrcq} also satisfies Robinson's CQ, so it can also be used to show that Robinson's CQ does not imply CRCQ. The main advantage of CRSC with respect to CRCQ is the fact that is depends on a single well-defined face of $\C$ instead of all faces. Moreover, observe that the linearized cone $L_\Omega(\xb)$ is simply the feasible set of the ``linearized problem'':
\begin{equation}
\begin{array}{ll}
		\underset{d\in \R^n}{\nonumber\mbox{\textnormal{Minimize }}} 	& \langle \nabla f(\xb), d\rangle\\ 
		\label{linredNCP}\mbox{subject to}	& D\G(\xb)d\in \C
	\end{array}
\tag{LinRed-NCP}
\end{equation}
around $\xb$, and that $F_{J_-}$ is the minimal face of $\C$ that represents~\eqref{linredNCP}, in the sense that the constraint $D\G(\xb)d\in \C$ could be replaced with $D\G(\xb)d\in F_{J_-}$ without changing its solution set. There are theoretical algorithms to compute such a face which are commonly found in the \textit{facial reduction} literature -- see, for instance,~\cite{pataki2013}. We explore this connection a bit further in Section~\ref{sec:fr}.

Next, we present a useful lemma that, although seems merely technical at this point, will inspire a deeper discussion later in Section~\ref{sec:fr}. To introduce it, recall that $L_{\Omega}(\xb)= D\G(\xb)^{-1}(\C)$ is the linearized cone of both~\eqref{NCP} and
 its reduced version~\eqref{redNCP} at $\xb$; the lemma states that $\C$ may be replaced with $F_{J_-}$ in the definition of $L_{\Omega}(\xb)$.

\begin{lemma}\label{lem:frlin}
Let $\xb\in \Omega$ and $F_{J_-}\doteq \Fmin(\Gamma_{\C}(\xb))\faceq \C$. Then, $L_{\Omega}(\xb)=D\G(\xb)^{-1}(F_{J_-})$.
\end{lemma}

\begin{proof}
By definition, it holds that $\Gamma_{\C}(\xb)\subseteq F_{J_-}$, which implies
\[
	D\G(\xb)^{-1}(D\G(\xb)[D\G(\xb)^{-1}(\C)])\subseteq D\G(\xb)^{-1}(F_{J_-});
\] and on the other hand we have
\[
	D\G(\xb)^{-1}(\C)\subseteq D\G(\xb)^{-1}(D\G(\xb)[D\G(\xb)^{-1}(\C)]),
\]
which holds trivially. 
 Combining the above statements, we get $L_{\Omega}(\xb)=D\G(\xb)^{-1}(\C)\subseteq D\G(\xb)^{-1}(F_{J_-})$, and the reverse inclusion follows directly from the fact $F_{J_-}\subseteq \C$. 
\end{proof}

In order to prove that CRSC (and, consequently, CRCQ) is a constraint qualification for~\eqref{NCP}, we will need some tools. The first one is an extension of a result by Andreani et al.~\cite[Proposition 3.1]{aes2010}, which is in turn a simplified version of a result originally presented by Janin~\cite[Proposition 2.2]{crcq}. The statement is as follows:

\begin{proposition}(\cite[Proposition 3.1]{aes2010})\label{prop1} Let $\{\zeta_{i}(x)\}_{i \in \mathcal{I}}$ be a finite family of differentiable functions $\zeta_i\colon \R^n\to \R$, $i\in \mathcal{I}$, such that the family of its gradients $\{\nabla \zeta_{i}(x)\}_{i \in \mathcal{I}}$ remains with constant rank in a neighborhood of $\xb$, and consider the linear subspace 
\begin{center}
$\mathcal{S} \doteq \left\{y \in \R^{n}\colon \langle \nabla \zeta_{i}(x), y\rangle = 0,  \ i \in \mathcal{I}\right\}$.
\end{center}
Then, there exists some neighborhoods $V_{1}$ and $V_{2}$ of $\xb$, and a diffeomorphism $\psi: V_{1} \rightarrow V_{2}$, such that:

\begin{itemize}
\item[(i)] $\psi(\xb) = \xb$;
\item[(ii)] $D\psi(\xb)=\mathbb{I}_n$;
\item[(iii)] $\zeta_{i}(\psi^{-1}(\xb+y))=\zeta_{i}(\psi^{-1}(\xb))$ for every $y \in \mathcal{S}\cap (V_2-\xb)$ and every $i\in \mathcal{I}$.
\end{itemize}
\end{proposition}

To make this result suitable for the conic environment, we will extend it as follows:

\begin{lemma}[Curve builder]\label{extra:curvebuilder}
Let $\G\colon \R^n\to \sF$ be twice differentiable, and let $\mathcal{W}\subseteq \sF$ be any subspace with dimension $N$. Also, let $\xb,d\in \R^n$ be such that $D\G(\xb)d\in \mathcal{W}^\perp$. If there exists a neighborhood $\mathcal{V}$ of $\xb$ such that $D\G(x)\adj[\mathcal{W}]$ remains with constant dimension for every $x\in \mathcal{V}$, then there exists a twice differentiable curve $\xi\colon \R\to \R^n$ such that $\langle \G(\xi(t)), V\rangle=\langle \G(\xb), V\rangle$ for every $t$ small enough and every $V\in \mathcal{W}$; moreover, $\xi(0)=\xb$ and $\xi'(0)=d$.
\end{lemma}

\begin{proof}
Let $\eta_1,\ldots,\eta_N$ be a basis of $\mathcal{W}$, and note that
\begin{equation}\label{eq:rank}
	D\G(x)\adj [\mathcal{W}]=
	\spn\left(\left\{
		D\G(x)\adj[\eta_i]\right\}_{i\in \{1,\ldots,N\}}\right).
\end{equation}
Therefore, the hypothesis on the constant dimension of $D\G(x)\adj[\mathcal{W}]$ can be equivalently stated as the constant rank of the family 
\[\left\{
		D{\G}(x)\adj[\eta_i]\right\}_{i\in \{1,\ldots,N\}}\]
for $x$ in a neighborhood of $\xb$. Furthermore, let 
\[
	\zeta_i(x)\doteq\langle \G(x),\eta_i\rangle
\]
and note that
\[
	\nabla \zeta_i(x)=D{\G}(x)\adj[\eta_i]
\] for every $i\in \{1,\ldots,N\}$. Then, by Proposition~\ref{prop1}, there exist neighborhoods $V_1$ and $V_2$ of $\xb$, and a curve $\psi\colon V_1\to V_2$ such that:
 \begin{itemize}
 \item $\psi(\xb)=\xb$;
 \item $D\psi(\xb)=\I_n$;
 \item $\zeta_i(\psi^{-1}(\xb+y))=\zeta_i(\xb)$ for every $i\in \{1,\ldots,N\}$ and every $y\in \mathcal{S}\cap (V_2-\xb)$;
 \end{itemize}
 where 
\[
	\mathcal{S}\doteq\left\{
		y\in \R^n\colon \langle \nabla \zeta_i(\xb), y\rangle=0, \ \forall i\in \{1,\ldots,N\}
	\right\}.
\]
Since $D{\G}(\xb)d\in \mathcal{W}^\perp$, we see that 
\[
	\langle d, D{\G}(\xb)\adj[\eta_i]\rangle=\langle D{\G}(\xb)d,\eta_i\rangle=0
\]
for every $i\in\{1,\ldots,N\}$, so $d\in \mathcal{S}.$ Then, define $\xi(t)\doteq\psi^{-1}(\xb+td)$ for every $t$ such that $\xb+td\in V_2$, and note that $\xi'(0)=d$ and $\xi(0)=\xb$. For every such $t$, we have
$\langle{\G}(\xi(t)),\eta_i\rangle=\langle{\G}(\xb),\eta_i\rangle$ for every $i\in \{1,\ldots,N\}$. Moreover, the degree of differentiability of $\xi$ is the same as the one of $\G$, which is a fact that follows from~\cite[Page 328]{minch}.
\end{proof}

%
%

Now, we recall a classical lemma:

\begin{lemma}[Proposition 2.1.12 of~\cite{urruty}]\label{lem:preimage-ri}
Let $\mathcal{L}$ be a linear mapping and let $C$ be a convex set such that $\mathcal{L}^{-1}(\ri(C))\neq \emptyset$. Then, $\ri(\mathcal{L}^{-1}(C))=\mathcal{L}^{-1}(\ri(C))$.
\end{lemma}

%

The previous lemma allows us to compute the relative interior of $L_{\Omega}(\xb)$ by considering $\mathcal{L}=D\G(\xb)$ and $C=F_{J_-}$. Indeed, this is possible due to the next lemma:

\begin{lemma}\label{lem:crsc-relaxrob}
For every $\xb\in \Omega$, there exists some $d\in D\G(\xb)^{-1}(\C)$ such that $D\G(\xb)d\in \ri(F_{J_-})$.
\end{lemma}
\begin{proof}
Recall from the definition that $F_{J_-}$ is the smallest face of $\C$ that contains $\Gamma_{\C}(\xb)$, which means that $\ri(\Gamma_{\C}(\xb))\cap \ri(F_{J_-})\neq \emptyset$ (see Lemma~\ref{lem:pataki} item 1). Therefore, there exists some $d\in D\G(\xb)^{-1}(\C)$ such that $D\G(\xb)d\in \ri(F_{J_-})$. 
\end{proof}

Note that Lemma~\ref{lem:crsc-relaxrob} tells us that Slater's CQ holds for the constraint $D\G(\xb)d\in F_{J_-}$ regardless of its fulfilment for the original linearized constraint $D\G(\xb)d\in \C$. This fact, together with Lemma~\ref{lem:frlin}, suggests that $F_{J_-}$ has a special property that we will explore in Section~\ref{sec:fr}. For now, as a consequence of the previous two lemmas, we obtain:

\begin{lemma}\label{lem:riL}
Let $\xb\in \Omega$. Then, $\ri(L_{\Omega}(\xb))= D\G(\xb)^{-1}(\ri(F_{J_-}))$.
\end{lemma}
\begin{proof}

Note that  $D\G(\xb)^{-1}(\ri(F_{J_-}))\neq \emptyset$ thanks to  Lemma~\ref{lem:crsc-relaxrob}.
Then, from Lemmas~\ref{lem:frlin} and~\ref{lem:preimage-ri}, 
we get 
\[
	\ri(L_{\Omega}(\xb))=\ri(D\G(\xb)^{-1}(F_{J_-}))= D\G(\xb)^{-1}(\ri(F_{J_-})),
\] which is the desired result.
\end{proof}

With these results at hand, it is simple to prove that CRSC is a constraint qualification:

\begin{theorem}\label{thm:crsccq}
Let $\xb\in \Omega$ be such that the dimension of $D\G(x)\adj[F_{J_-}^\perp]$ is constant for $x$ in a neighborhood of $\xb$. Then, $T_{\Omega}(\xb)=L_{\Omega}(\xb)$.
\end{theorem}
\begin{proof}
First, note that $T_{\Omega}(\xb)\subseteq L_{\Omega}(\xb)$.  Then, since $T_{\Omega}(\xb)$ is closed, it suffices to prove that $\ri(L_{\Omega}(\xb))\subseteq T_{\Omega}(\xb)$ to conclude that $L_{\Omega}(\xb)\subseteq T_{\Omega}(\xb)$, and consequently the desired equality. 
Let $d\in \ri(L_{\Omega}(\xb))$ and we have $D\G(\xb)d\in \ri(F_{J_-})$ due to Lemma~\ref{lem:riL}. By Lemma~\ref{extra:curvebuilder} with $\mathcal{W}= F_{J_-}^\perp$ there exists some $\varepsilon>0$ and a curve $\xi\colon (-\varepsilon,\varepsilon)\to \R^n$ such that $\xi(0)=\xb$, $\xi'(0)=d$, and 
\[
	\langle \G(\xi(t)),V\rangle=\langle \G(\xb),V\rangle=0
\]
for every $V\in F_{J_-}^\perp$ and every $t\in (-\varepsilon,\varepsilon)$. That is, $\G(\xi(t))\in \spn(F_{J_-})$ for every such $t$. Taking the Taylor expansion of $\G(\xi(t))$, we see that
\[
	\G(\xi(t))=tD\G(\xb)d+o(t),
\]
but since $D\G(\xb)d\in \ri(F_{J_-})$, we have $\G(\xi(t))\in \ri(F_{J_-})$ as well, for every $t\in(-\varepsilon,\varepsilon)$ shrinking $\varepsilon$ if necessary.
Hence, $d\in T_{\Omega}(\xb)$, but  since $d$ is arbitrary, it follows that $\ri(L_{\Omega}(\xb))\subseteq T_{\Omega}(\xb)$.
\end{proof}

Since the tangent cone of~\eqref{NCP} coincides with the tangent cone of~\eqref{redNCP} at any $\xb\in \F$ after applying its respective reduction mapping, the condition
\begin{center}
	$T_{\Omega}(\xb)=\mathcal{L}(\xb)$ and $H(\xb)$ is closed
\end{center}
characterizes a constraint qualification known as \textit{Abadie's CQ} for~\eqref{redNCP}.
Thus, Theorem~\ref{thm:crsccq} tells us that CRSC is a constraint qualification for the original problem~\eqref{NCP} because of the correspondence between Lagrange multipliers of~\eqref{NCP} and~\eqref{redNCP}, and since CRCQ implies CRSC, we obtain as a corollary that CRCQ is also a constraint qualification.

\begin{corollary}
Let $\xb\in \Omega$ be a local minimizer of~\eqref{NCP} that satisfies CRSC (or CRCQ). Then $\xb$ satisfies the KKT condition for~\eqref{NCP}.
\end{corollary}

Furthermore, Robinson's CQ also implies CRSC.

\begin{proposition}\label{prop:robthencrsc}
Let $\xb\in \Omega$ satisfy Robinson's CQ. Then, $\xb$ satisfies CRSC.
\end{proposition}
\begin{proof}
If Robinson's CQ holds at $\xb$, then there exists some $d\in \R^n$ such that $D\G(\xb)d\in \int(\C)$, which implies that
\[
	D\G(\xb)[D\G(\xb)^{-1}(\int(\C))]\cap \int(\C)\neq\emptyset,
\]
so $\C=F_{J_-}$ and $D\G(x)\adj[F_{J_-}^\perp]=\{0\}$, which has constant rank (equal to zero) for every $x\in\R^n$. The proof of the closedness of $H(\xb)=DG(\xb)^*[T_{\K}(G(\xb))\pol]$ follows directly from two well-known statements:
\begin{itemize}
\item \cite[Corollary 2.98]{bonnans-shapiro} Robinson's CQ implies that \[\Ker (DG(\xb)^*) \cap T_{\K}(G(\xb))\pol = \{0\};\]
\item \cite[Theorem 9.1]{rockafellar1970convex} Let $\mathbb{X}$ and $\mathbb{Y}$ be finite-dimensional linear spaces, $\emptyset\neq C\in \mathbb{X}$ be a closed convex cone, and $A\colon \mathbb{X}\to \mathbb{Y}$ be a linear mapping. If $C\cap \Ker A = \{0\}$, then $A[C]$ is closed.
\end{itemize} 
This concludes the proof.
\end{proof}

In particular,  To further clarify the relationship between CRCQ and CRSC, we shall present a small lemma followed by an example that shows that CRSC does not imply Robinson's CQ.

\begin{lemma}\label{lem:hclosed}
Let $\xb\in \Omega$ and let $\mathcal{F}(\xb)$ be the set of all objective functions that are continuous in a neighborhood of $\xb$ and have a local minimum constrained by $\Omega$ at $\xb$. If the KKT conditions hold at $\xb$ for every $f\in\mathcal{F}(\xb)$, then $H(\xb)$ is closed.
\end{lemma}
\begin{proof}
Let $D\mathcal{F}(\xb)\doteq \left\{-\nabla f(\xb)\in \R^n\colon f\in \mathcal{F}(\xb) \right\}$ and recall~\cite[Corollary 3.4]{gould1975optimality} which states that \[
D\mathcal{F}(\xb)=\cl(H(\xb))
\] where $\cl(H(\xb))$ stands for the closure of $H(\xb)$. By hypothesis, we have that $D\mathcal{F}(\xb)\subseteq H(\xb)$, thus $H(\xb)$ is closed.
\end{proof}

\begin{example}
Consider an NSDP problem constrained by:
\[
	G(x)\doteq \begin{bmatrix}
	x_1 & x_2-x_1^2 & x_1 \\
	x_2-x_1^2 & x_2 & x_2-x_1^2\\
	x_1 & x_2-x_1^2 & x_1\\
	\end{bmatrix}\in\S^3_+.\\
\]
We will perform the analysis at the feasible point $\xb=(0,0)\in \R^2$. Note that it admits the trivial reduction via the identity mapping at $\xb$, that is, $\mathcal{G}=G$ and $\mathcal{C}=\S^{3}_+$. The linearized cone can be computed as follows:
\[
	L_{\Omega}(\xb) = \{(d_1,d_2)\in \R^2\colon 0\leq d_2\leq d_1\}.
\]
Now let $d\doteq (d_1,d_2)$ and \[
	U\doteq \begin{bmatrix}
	\frac{1}{\sqrt{2}} & 0 & \frac{-1}{\sqrt{2}}\\
	0 & 1 & 0\\
	\frac{1}{\sqrt{2}} & 0 & \frac{1}{\sqrt{2}}\\
	\end{bmatrix}
	\
	\textnormal{and notice that} 
	\
	U^\top D\G(\xb)d U=\begin{bmatrix}
	2d_1 & \sqrt{2}d_2 & 0\\
	\sqrt{2}d_2 & d_2 & 0\\
	0 & 0 & 0
	\end{bmatrix}
\]
then
\[
	F_{J_-} = F_{\min}(D\G(\xb)[L_{\Omega}(\xb)]) = \left\{
	U	
	\begin{bmatrix}
	A & 0\\
	0 & 0
	\end{bmatrix}	
	U^\top\colon A\in \S^2_+	
	\right\}
\]
and, consequently,
\[
	F_{J_-}^\perp = \left\{
	U	
	\begin{bmatrix}
	0 & 0 & m_{13}\\
	0 & 0 & m_{23}\\
	m_{13} & m_{23} & m_{23}\\
	\end{bmatrix}	
	U^\top\colon m_{i,j}\in \R	
	\right\}.
\]
Now, because for every $M\doteq [m_{ij}]_{i,j\in\{1,3\}}\in \R^{3\times 3}$ we have
\[
	D\G(x)^*[UMU^\top]=\left\{\begin{bmatrix}
	2m_{11}-2\sqrt{2}x_1(m_{12}+m_{21})\\
	m_{22}+\sqrt{2}(m_{12}+m_{21})
	\end{bmatrix}\in \R^2\colon m_{ij}\in\R\right\}
\]
it follows that $D\G(x)^*[F_{J_-}^\perp]=\{0\}$ for every $x\in \R^2$, meaning its dimension remains constant for every $x$. To prove that $H(\xb)$ is closed, let $f\in \mathcal{F}(\xb)$, where $\mathcal{F}(\xb)$ is defined as in Lemma~\ref{lem:hclosed}, and suppose that $\nabla f(\xb)\doteq(\alpha, \beta)\in \R^2$. Notice that
\[
	G(x)\succeq 0 \quad \text{if, and only if,} \quad U^\T G(x)U = \begin{bmatrix}
		2x_1 & 2(x_2 - x_1^2) & 0 \\
    		2(x_2 - x_1^2) & x_2 & 0 \\
    		0 & 0 & 0
	\end{bmatrix}\succeq 0
\]
so the feasible set can be described as follows:
\[
\Omega= \left\{x\in \R^2 \colon 
\begin{array}{c}
	x_1\geq 0, \ x_2 \geq 0 \\
	x_1 x_2 - 2(x_2 - x_1^2)^2 \geq 0)
\end{array}\right\}.
\]
Observe that for every $t>0$ the curve $\phi\colon t\mapsto (t,t^2)$ is feasible. This implies that $\alpha \geq 0$, because otherwise, if $\alpha<0$, we would have $f(\phi(t))\approx f(\xb) + \alpha t + \beta t^2<f(\xb)$ for $t$ small enough, and $\xb$ could not be a local minimizer for $f$. Now, we have two cases to analyse:
\begin{itemize}
\item If $\beta \geq 0$, then take 
\[
\Lambda \doteq \begin{bmatrix}
 				\alpha  & 0 & 0 \\
 				 0& \beta & 0 \\
 				0 & 0 & 0
 			\end{bmatrix}\in \S^m_-
\] which satisfies $\nabla f(\xb) + DG(\xb)^*[\Lambda]=0$;
\item If $\beta < 0$, observe that the curve $\psi\colon t\mapsto (t,t)$ is also feasible for $t\in[0,2)$, and since $\xb$ is a local minimizer of $f$ we must have $\alpha>-\beta>0$, so let $\delta \doteq 1-\frac{\beta}{\alpha}$, which satisfies $|\delta| < 2$. This implies that
\[
\Lambda \doteq \alpha\begin{bmatrix}
	-\frac{3}{2} & \frac{\delta}{4} & 1 \\
	\frac{\delta}{4} & -1 & \frac{\delta}{4} \\
	1 & \frac{\delta}{4} & -\frac{3}{2}
\end{bmatrix}\in \S^m_-
\]
and for this choice of $\Lambda$ we have $\nabla f(\xb) + DG(\xb)^*[\Lambda]=0$ also.
\end{itemize}
Therefore, $\xb$ satisfies the KKT conditions for all $f\in \mathcal{F}(\xb)$ and, by Lemma~\ref{lem:hclosed}, it follows that $H(\xb)$ is closed and thus $\xb$ satisfies CRSC. However, notice that $D\G(\xb)d$ always has one zero eigenvalue for every $d\in \R^2$, so Robinson's CQ does not hold at $\xb$.
\end{example}

There are also counterexamples in NLP presented by Janin~\cite[Examples 2.1 and 2.2]{crcq} that, just as the previous example, illustrate that CRCQ is independent of Robinson's CQ. 
%

\section{Second-order analysis}

In a previous work by Andreani et al. \cite{CRCQfaces} for NSDP and NSOCP, the authors used CRCQ to obtain a strong second-order optimality condition for~\eqref{NCP} depending on any single given Lagrange multiplier. This result is, in particular, stronger than the classical second-order condition of Bonnans, Cominetti, and Shapiro~\cite{bonn-comi-shap}. However, since Robinson's CQ is not enough for obtaining a similar result \cite{conjnino}, the same holds for CRSC. We can, though, define a new constraint qualification for~\eqref{NCP} that is in-between CRCQ and CRSC, which can be used to obtain the strong second-order condition of~\cite{CRCQfaces} by directly applying CRCQ to the constraint $\G(x)\in F_{J_-}$ (see also Section~\ref{sec:fr}). That is:

\begin{definition}[Strong-CRSC]\label{def:scrsc}
Let $\xb\in \Omega$ and define $F_{J_-}\doteq \Fmin(D\G(\xb)[L_\Omega(\xb)])\faceq \C$. We say that \emph{Strong-CRSC} holds at $\xb$ when $H(\xb)$ is closed and there exists a neighborhood $\mathcal{V}$ of $\xb$ such that for each $F\faceq F_{J_-}$, the dimension of $D\G(x)^*[F^\perp]$ remains constant for every $x\in \mathcal{V}$.
\end{definition}

Observe that Strong-CRSC implies CRSC by taking $F=F_{J_-}$, and note also that CRCQ implies Strong-CRSC since all faces of $F_{J_-}$ are, in particular, faces of $\C$. 
On the other hand, Example~\ref{ex:crscnotcrcq} proves that Strong-CRSC is still strictly weaker than CRCQ since in that example $F_{J_-}=\{0\}$ and thus Strong-CRSC coincides with CRSC there. Finally, it is worth mentioning that Strong-CRSC is new even in NLP, and any second-order result involving it is also an improvement of the existing NLP results (which are based on CRCQ \cite{rcpld}). 

To proceed, let us prove the strong second-order condition under Strong-CRSC:

\begin{theorem}\label{ssoc-scrsc}
Let $\xb\in \F$ be a local minimizer of \eqref{NCP} that satisfies Strong-CRSC. Then, for every $\bar{Y}\in \Lambda(\xb)$ and every $d\in C(\xb)\doteq L_{\Omega}(\xb)\cap \{\nabla f(\xb)\}^\perp$, the following inequality is satisfied:
\begin{equation}\label{eq:ssoc}
	d^T \nabla^2 f(\xb)d+\left\langle D^2G(\xb)[d,d], \bar{Y} \right\rangle \geqslant \sigma(\xb,d,\bar{Y}),
\end{equation}
where
\[
	\sigma(\xb,d,\bar{Y})\doteq \sup\{\langle W, \bar{Y}\rangle\colon W\in T^2_{\K}(G(\xb),DG(\xb)d)\}
\]
and $T^2_{\K}(G(\xb),DG(\xb)d)$ is the second-order tangent set to $\K$ at $G(\xb)$ along $DG(\xb)d$.
\end{theorem}

\begin{proof}
Let $\bar{Y}\in \Lambda(\xb)$ and $d\in C(\xb)$ be arbitrary; then $D\G(\xb)d\in F_{J_-}$ thanks to Lemma~\ref{lem:frlin}. Moreover, let $F$ be the smallest face of $F_{J_-}$ that contains $D\G(\xb)d$ in its relative interior\footnote{\label{footnote:2}Noteworthy, if $d\in \ri(\mathcal{L}(\xb))$, this face is $F_{J_-}$ itself.}.
Then, by Strong-CRSC, similarly to the proof of Theorem~\ref{thm:crsccq}, there exists some $\varepsilon>0$ and a twice continuously differentiable curve $\xi\colon (-\varepsilon,\varepsilon)\to\R^n$ such that $\xi(0)=\xb$, $\xi'(0)=d$, and 
\[
	\mathcal{G}(\xi(t))\in \textnormal{ri}(F)
\]
for all $t\in[0,\varepsilon)$. Since $\xb$ is a local minimizer of~\eqref{NCP} and $\xi(t)$ is, in particular, feasible for every small $t$, then $t=0$ is a local minimizer of the function $\phi(t)\doteq f(\xi(t))$ subject to the constraint $t\geq 0$. Consequently,
\begin{equation}\label{sdp_phi}
\phi^{\prime\prime}(0)=d^T \nabla^2 f(\xb)d+\nabla f(\xb)^T\xi^{\prime\prime}(0)\geqslant 0.
\end{equation}	

The rest of the proof consists of computing the term $\nabla f(\xb)^T\xi^{\prime\prime}(0)$. To do so, let $\bar{\mathcal{Y}}$ be such that $\bar{Y}=D\Xi(G(\xb))^*[\bar{\mathcal{Y}}]$, which is uniquely determined since $D\Xi(G(\xb))^*$ is injective. By the KKT conditions, we have that
\[
	\langle d, \nabla f(\xb) \rangle=-\left\langle d, DG(\xb)\adj[\bar{Y}]\right\rangle=-\left\langle DG(\xb)d, \bar{Y}\right\rangle=0.
\]
Therefore,
\jo{\begin{eqnarray*}
	\langle DG(\xb)d,\bar{Y}\rangle=\langle DG(\xb)d,D\Xi(G(\xb))^*[\bar{\mathcal{Y}}]\rangle=\langle D\Xi(G(\xb))DG(\xb)d,\bar{\mathcal{Y}}\rangle\\
	=\langle D\G(\xb)d,\bar{\mathcal{Y}}\rangle=0,
\end{eqnarray*}}
\si{\[
	\langle DG(\xb)d,\bar{Y}\rangle=\langle DG(\xb)d,D\Xi(G(\xb))^*[\bar{\mathcal{Y}}]\rangle=\langle D\Xi(G(\xb))DG(\xb)d,\bar{\mathcal{Y}}\rangle=\langle D\G(\xb)d,\bar{\mathcal{Y}}\rangle=0,
\]}
so $D\G(\xb)d\in F\cap \{\bar{\mathcal{Y}}\}^\perp$. Note that $F\cap \{\bar{\mathcal{Y}}\}^\perp$ is a face of $F$ because $\bar{\mathcal{Y}}\in \C\pol$ and, consequently, $ \{\bar{\mathcal{Y}}\}^\perp$ is a supporting hyperplane to $F$. Then, $F\cap \{\bar{\mathcal{Y}}\}^\perp$ is also a face of $\C$ because $F$ is a face of $\C$. On the other hand, $F$ is by construction the minimal face of $\C$ containing $D\G(\xb)d$, so we must have $F=F\cap \{\bar{\mathcal{Y}}\}^\perp$, hence $F\subseteq \{\bar{\mathcal{Y}}\}^\perp$ and, consequently, $\bar{\mathcal{Y}}\in F^\perp$.
Then, consider the function
\[
	R(t)\doteq \left\langle\mathcal{G}(\xi(t)), \bar{\mathcal{Y}} \right\rangle
\]
and our previous reasoning implies that $R(t)=0$ for every small $t\geq 0$. Differentiating $R(t)$, we obtain:
\[
	\begin{aligned}
		R'(t) & = \left\langle D\Xi(G(\xi(t)))DG(\xi(t))\xi'(t), \bar{\mathcal{Y}} \right\rangle.\\
	\end{aligned}
\]
Differentiating it once more, at $t=0$, we obtain:
\jo{\[
	\begin{aligned}
		R^{\prime\prime}(0) & = \frac{d}{dt}\left\langle D\Xi(G(\xi(t)))DG(\xi(t))\xi'(t), \bar{\mathcal{Y}} \right\rangle\left|_{t=0}\right.\\
		& = \left\langle D^2\Xi(G(\xb))[DG(\xb)d,DG(\xb)d] + D\Xi(G(\xb))D^2G(\xb)[d,d], \bar{\mathcal{Y}} \right\rangle +\\
		 & \left\langle D\Xi(G(\xb))DG(\xb)\xi^{\prime\prime}(0), \bar{\mathcal{Y}} \right\rangle\\
		& = 0.
	\end{aligned}
\]}
\si{\[
	\begin{aligned}
		R^{\prime\prime}(0) & = \frac{d}{dt}\left\langle D\Xi(G(\xi(t)))DG(\xi(t))\xi'(t), \bar{\mathcal{Y}} \right\rangle\left|_{t=0}\right.\\
		& = \left\langle D^2\Xi(G(\xb))[DG(\xb)d,DG(\xb)d] + D\Xi(G(\xb))D^2G(\xb)[d,d]+ D\Xi(G(\xb))DG(\xb)\xi^{\prime\prime}(0), \bar{\mathcal{Y}} \right\rangle\\
		& = 0.
	\end{aligned}
\]}
Using the fact
\[
	\left\langle D\Xi(G(\xb))DG(\xb)\xi^{\prime\prime}(0), \bar{\mathcal{Y}} \right\rangle=\left\langle \xi^{\prime\prime}(0), DG(\xb)^*[\bar{Y}] \right\rangle=-\nabla f(\xb)^T\xi^{\prime\prime}(0),
\]
it follows that
\begin{equation}\label{eq:invariance}
	\nabla f(\xb)^T\xi^{\prime\prime}(0)=\left\langle D^2\Xi(G(\xb))[DG(\xb)d,DG(\xb)d],\bar{\mathcal{Y}} \right\rangle + \left\langle D^2G(\xb)[d,d], \bar{Y} \right\rangle.
\end{equation}
Moreover, following Bonnans and Shapiro~\cite[Equations 3.274 and 3.275]{bonnans-shapiro}, we see that
\[
	\sigma(\xb,d,\bar{Y})=-\left\langle D^2\Xi(G(\xb))[DG(\xb)d,DG(\xb)d],\bar{\mathcal{Y}} \right\rangle.
\]
Thus, substituting the above expressions in~\eqref{sdp_phi}, we conclude that
\begin{equation*}
	d^T \nabla^2 f(\xb)d+\left\langle D^2G(\xb)[d,d], \bar{Y} \right\rangle \geqslant \sigma(\xb,d,\bar{Y}).
\end{equation*}
Since $d$ and $\bar{Y}$ were chosen arbitrarily, this proof is complete.
\end{proof}

\begin{remark}Notice that the closedness of $H(\xb)$ (subsumed by the definition of Strong-CRSC) is actually not needed in the proof of Theorem \ref{ssoc-scrsc}.\end{remark}

Using only CRSC, we can obtain a weaker result with an analogous proof (see Footnote~\ref{footnote:2}). 
 What follows is a formal statement of such result:

\begin{corollary}
Let $\xb\in \F$ satisfy CRSC. Then,~\eqref{eq:ssoc} holds for every $\bar{Y}\in \Lambda(\xb)$ and every $d\in \ri(L_{\Omega}(\xb))\cap \{\nabla f(\xb)\}^\perp\subseteq C(\xb)$.
\end{corollary}

Another interesting result that can be extracted from the proof of the previous theorem is the invariance to Lagrange multiplier of the second-order term \eqref{eq:ssoc} induced by the Hessian of the Lagrangian for each critical direction, stated below. Variants of this result have also been proposed by \cite[Theorem 3.3]{conjnino} for the directions in $\lin(C(\xb))$ and by \cite{gfrerer} for all directions in $\textnormal{span}(C(\xb))$.

\begin{proposition}\label{invariance}
Let $\xb\in \Omega$ satisfy Strong-CRSC. Then, for each $d\in C(\xb)$, the second-order term \eqref{eq:ssoc} is independent of the choice of the Lagrange multiplier $\bar{Y}\in \Lambda(\xb)$, that is, for each $d\in C(\xb)$
\begin{equation}\label{eq-invariance}
	d^T\nabla^2 f(\xb)d+\left\langle D^2 G(\xb)[d,d], \bar{Y}\right\rangle - \sigma(\xb,d,\bar{Y})\mbox{ does not depend on }\bar{Y}.
\end{equation}
\end{proposition}
\begin{proof}
The conclusion follows directly from~\eqref{eq:invariance}.
\end{proof}

Notice that if we assume that the dimension of $D\G(x)^*[F^\perp]$ remains constant for all $x$ in a neighborhood of $\xb$, only for the face $F=\{0\}\in \C$, then~\eqref{eq-invariance} holds true at least for every $d\in\lin(C(\xb))=\{d\in \R^n\colon D\G(\xb)d=0\}$, which generalizes~\cite[Theorem 3.3]{conjnino}.

\section{Facial reduction for nonconvex optimization problems}\label{sec:fr}

Facial reduction is a preprocessing technique originally introduced for the convex case of~\eqref{NCP} by Wolkowicz and Borwein~\cite{wolkowicz1981}, which stood out for inducing strong duality results without constraint qualifications in a mathematically elegant way. Their work was later revisited and improved by Pataki~\cite{pataki2013}, and Waki and Muramatsu~\cite{waki2013}, who provided a very simple derivation for the facial reduction algorithm and its underlying results regarding linear conic problems. 

More precisely, let us assume, for a moment, that the constraint of~\eqref{NCP} has the form $G(x)\doteq Ax+B\in \K$, where $A\colon \R^n\to \E$ is a linear operator and $B\in \E$. Let $\Omega\doteq G^{-1}(\K)\neq\emptyset$ and $\Gamma\doteq \{Ax+B:x\in\Omega\}$. Then $Ax+B\in \K \ \Leftrightarrow \ Ax+B\in \Fmin(\Gamma)$, where the latter satisfies Slater's constraint qualification. Interestingly, in the case of a linear problem, the minimal face $\Fmin(\Gamma)$ can be iteratively computed by considering $F\doteq\K$ and computing $S\in F\pol\cap \Ker A^*\cap \{B\}^\perp$ in order to have $F\cap\{S\}^\perp\faceq\K$ with $\Fmin(\Gamma)\subseteq F\cap\{S\}^\perp$. Now the procedure can be repeated for $F\doteq F\cap\{S\}^\perp$ until $F=\Fmin(\Gamma)$. See \cite{pataki2013} for details.

Our goal is to show that in the general case of \eqref{NCP}, under a set of hypotheses, condition CRSC gives the appropriate tool for allowing the problem to be locally rewritten by replacing the cone with one of its faces, namely:
\begin{equation}
	\begin{array}{ll}
		\nonumber\mbox{\textnormal{Minimize }} 	& f(x),\\ 
		\label{rNCP}\mbox{s.t. }	& \G(x)\in F_{J_-}.
	\end{array}
	\tag{FRed-NCP}
\end{equation}
Recall that Lemma~\ref{lem:riL} tells us that the linearized constraint $D\G(x)d\in \C$ at a given point $\xb\in \F$ coincides with the ``facially reduced'' constraint $D\G(x)d\in F_{J_-}$ at $\xb$, which in turn always satisfies Slater's CQ due to Lemma~\ref{lem:crsc-relaxrob}. Notice that this also implies that Robinson's condition is satisfied at $\xb$ with respect to problem \eqref{rNCP}. 

The extension of the facial reduction procedure to the nonlinear case has been already conducted under CRSC in \cite{cpg} in the context of NLP, which was obtained as a consequence of a result by Lu~\cite{shulu} regarding CRCQ. This interpretation of their result is what motivates the course of this section. Let us start by recalling Lu's result, as stated in~\cite{cpg}:

\begin{proposition}\label{prop:lu}
Let $\zeta_i\colon \R^n\to \R$, $i=1,\dots,s$ be continuously differentiable functions and $\xb$ be such that $\zeta_i(\xb)=0$ for every $i=1,\dots,s$. Assume that the family $\{\nabla \zeta_i(x)\}_{i=1}^s$ has constant rank for all $x$ in a neighborhood of $\xb$. If there exists $\gamma\in\R^s$ with $\gamma_i>0$ for all $i=1,\dots,s$ such that
\[
	\sum_{i=1}^s \nabla \zeta_i(\xb)\gamma_i=0,
\]
then there exists a neighborhood $\mathcal{U}$ of $\xb$ such that for every $i=1,\dots,s$ and every $x\in \mathcal{U}$, we have that $\zeta_i(x)\geqslant0$ if, and only if, $\zeta_i(x)=0$.
\end{proposition}

In order to adapt this proposition the the conic context, let us consider first the following well-known result regarding a characterization of the relative interior. We present a proof\footnote{Thanks to Joe Higgins for providing this ingenious proof.} for completeness:

\begin{lemma}\label{lem:carath2}
Let $C\subseteq\E$ be a closed convex cone with dimension $s\geqslant 1$, and let $0\neq Y\in \ri(C)$. Then, there exist some linearly independent vectors $\eta_1,\ldots,\eta_s\in C$ such that $Y=\sum_{i=1}^s\alpha_i \eta_i$ with $\alpha_i>0$ for every $i\in \{1,\ldots,s\}$.
\end{lemma}

\begin{proof}
Let $Z_1,\ldots,Z_{s-1}$ be a basis of $\spn(C)\cap \{Y\}^\perp$ and define $Z_s\doteq -(Z_1+\ldots+Z_{s-1})$. Now, let $\varepsilon>0$ be such that $\eta_i\doteq Y+\varepsilon Z_i\in \ri(C)$ for every $i\in \{1,\ldots,s\}$ and note that $\eta_1,\ldots,\eta_s$ are linearly independent, since
\[
	\beta_1 \eta_1+\ldots+\beta_s \eta_s=(\beta_1+\ldots+\beta_s)Y+\varepsilon(\beta_1-\beta_s)Z_1+\ldots+\varepsilon(\beta_{s-1}-\beta_s) Z_{s-1}=0
\]
implies $\beta_1=\ldots=\beta_{s-1}=\beta_s=0$. Define $\alpha_i\doteq 1/s>0$ for every $i\in \{1,\ldots,s\}$ to conclude the proof.
\end{proof}

To state the generalization of Proposition~\ref{prop:lu} to the conic programming context, we define the {\it conjugate face} $F^{\triangle}$ of a face $F\faceq\C$ as $F^{\triangle}\doteq-\C\pol\cap F^\perp$. Following Pataki~\cite{pataki2007closedness}, for a face $H\faceq -\C\pol$ we use the same notation $H^{\triangle}$ to describe the conjugate face $H^{\triangle}\doteq \C\cap H^\perp$ and the face $F^{\triangle\triangle}$ is called the \textit{double conjugate} of $F$. Moreover, if $\C$ is {\it nice}, that is, if $\C\pol+F^\perp$ is closed for every $F\faceq\C$, then $F^{\triangle\triangle}=F$ for every face $F\faceq \C$~\cite{pataki2007closedness,pataki2013}, so through the following lines we shall assume that $\C$ is nice. Recall that the second-order and the semidefinite cones are nice -- see, for instance, \cite[Section 2.5]{pataki2007closedness}.

\begin{proposition}\label{prop:frlu}
Let $\xb\in \Omega$ and let $F\faceq \C$ be such that:
\begin{enumerate}
\item There exists some $0\neq Y\in \ri(F^{\triangle})$ such that $D\G(\xb)^*[Y]=0$;
\item The dimension of $D\G(x)\adj[\spn(F^{\triangle})]$ remains constant in a neighborhood $\mathcal{V}$ of $\xb$.
\end{enumerate}
Then, there exists a neighborhood $\mathcal{U}$ of $\xb$ such that $\G^{-1}(\C)\cap \mathcal{U}=\G^{-1}(F)\cap \mathcal{U}$.
\end{proposition}
%
%
%
%

\begin{proof}
Let $s\doteq \dim(\spn(F^{\triangle}))$ and let $0\neq Y\in \ri(F^{\triangle})$ be such that $D\G(\xb)^*[Y]=0$; this implies that $s\geqslant 1$. Also, let $\eta_1,\ldots,\eta_s\in F^{\triangle}$ be the vectors described in Lemma~\ref{lem:carath2}, which form a basis of $\spn(F^{\triangle})$ and
\[
	Y=\sum_{i=1}^s\alpha_i \eta_i,
\]
for some vector of positive scalars $\alpha\doteq(\alpha_1,\ldots,\alpha_s)$. Now, consider the functions
\[
	\zeta_i(x)\doteq \langle \G(x),\eta_i\rangle, \quad i\in \{1,\ldots,s\}
\]
along with the vector function $\zeta(x)\doteq (\zeta_1(x),\ldots,\zeta_s(x))$, and note that $\nabla \zeta_i(x)=D\G(x)^*[\eta_i]$ for every $i\in\{1,\ldots,s\}$. In particular, this implies that 
\[
	\rank(\{\nabla \zeta_i(x)\}_{i\in\{1,\ldots,s\}})=\dim(D\G(x)^*[\spn(F^{\triangle})])
\] for every $x\in \R^n$ and moreover,
\begin{equation}\label{eq:gradzeta}
D\zeta(\xb)^\T\alpha=D\G(\xb)^*[Y]=0.
\end{equation}
Applying Proposition~\ref{prop:lu}, we obtain a neighborhood $\mathcal{U}$ of $\xb$ such that $\zeta(x)=0$ for every $x\in \mathcal{U}$ such that $\zeta(x)\geqslant 0$. However, note that $\eta_i\in -\C\pol$ for every $i\in\{1,\ldots,s\}$, so in particular this holds for every $x\in \Omega\cap \mathcal{U}$ because
\[
	\begin{aligned}
	\Omega 
	& = \left\{x\in \R^n\colon \langle \G(x),\eta\rangle\geqslant 0, \ \forall \eta\in -\C\pol\right\}\\	
	&  \subseteq \left\{x\in \R^n\colon \zeta(x)\geqslant 0\right\}.
	\end{aligned}
\]
Furthermore, since $\zeta(x)=0$ is equivalent to $\G(x)\in (F^{\triangle})^{\perp}$ we use the niceness of $\C$ to obtain $(F^{\triangle})^{\perp}\cap \C=F^{\triangle\triangle}=F$ and conclude that
\[
	\G^{-1}(\C)\cap \mathcal{U}=\G^{-1}((F^{\triangle})^{\perp}\cap \C)\cap \mathcal{U}=\G^{-1}(F)\cap \mathcal{U}.
\]
\end{proof}

In order to connect this result with CRSC, we will make use of the following characterization of the closedness of the image of a closed convex cone by a linear operator in terms of minimal faces:

\begin{theorem}[Theorem 1 of~\cite{pataki2007closedness}]\label{thm:closedness}
Let $C\subseteq\mathbb{F}$ be a nice closed convex cone, let $M:\E\to\mathbb{F}$ be a linear operator, and $F\doteq \Fmin(\Im (M)\cap C)$. The following statements are equivalent:
\begin{itemize}
\item $M^*[C\pol]$ is closed;
\item $\ri(F^{\triangle})\cap \Ker(M^*)\neq \emptyset$ and $\Im(M) \cap (F^{\triangle})^{\perp}=\Im(M)\cap \spn(F)$;
\item $M^*[F^{\triangle}]=M^*[F^\perp]$.
\end{itemize}
\end{theorem}

We are now finally able to state the main result of this section, which is a generalization of the result of \cite{cpg}, and which can be interpreted as a facial reduction for the general \eqref{NCP}:

\begin{theorem}\label{ncp:frcrsc}
Let $\xb\in \Omega$ satisfy CRSC, suppose that $\C$ is nice, and let $F_{J_-}\doteq \Fmin(\Gamma_{\C}(\xb))$, with $\Gamma_{\C}(\xb)\doteq D\G(\xb)[L_{\Omega}(\xb)]$ and $L_{\Omega}(\xb)=\{d\in \R^n\colon D\G(\xb)[d]\in \C\}$. Assume that there exists a neighborhood $\mathcal{V}$ of $\xb$ such that:
\begin{itemize}
\item [\textbf{A1.}] $H(x)\doteq D\G(x)^*[\C\pol]$ is closed for every $x\in\mathcal{V}$;
\item [\textbf{A2.}] $F_{J_-}=\Fmin(\Gamma_{\C}(x))$, for every $x\in \mathcal{V}$.
\end{itemize}
Then, there exists a neighborhood $\mathcal{U}$ of $\xb$ such that
\[
	\G^{-1}(\C)\cap \mathcal{U}=\G^{-1}(F_{J_-})\cap \mathcal{U}.
\]
\end{theorem}

\begin{proof}
First, note that
\[
	\Gamma_{\C}(\xb)=D\G(\xb)[D\G(\xb)^{-1}(\C)]=\Im(D\G(\xb))\cap \C
\]
so applying Theorem~\ref{thm:closedness} with $M\doteq D\G(\xb)$, $C\doteq \C$, and $F\doteq F_{J_-}$, we see that there exists some $0\neq Y\in \ri(F_{J_-}^{\triangle})$ such that $D\G(\xb)^*[Y]=0$ because in this case $M^*[C\pol]=H(\xb)$ is closed. By CRSC the dimension of $D\G(x)^*[F_{J_-}^{\perp}]$ remains constant for every $x$ in a neighborhood of $\xb$, which we can assume to be $\mathcal{V}$ without loss of generality. On the other hand, by A1 and A2, $D\G(x)^*[\C\pol]$ is closed for every $x\in \mathcal{V}$ and $F_{J_-}$ is the minimal face of $\Gamma_{\C}(x)$, therefore the dimension of $D\G(x)^*[\spn(F_{J_-}^{\triangle})]$ also remains constant for every $x\in\mathcal{V}$ since by Theorem~\ref{thm:closedness} we have that \[D\G(x)^*[F_{J_-}^{\triangle}] = D\G(x)^*[F_{J_-}^{\perp}]\] for every such $x$; Proposition~\ref{prop:frlu} then gives us a neighborhood $\mathcal{U}$ of $\xb$ such that $\G^{-1}(\C)\cap \mathcal{U}=\G^{-1}(F_{J_-})\cap \mathcal{U}$.
\end{proof}

It is worth noticing that in NLP assumptions A1 and A2 are not required due to the polyhedricity of $\C$ and \cite[Lemma 5.4]{cpg}. Whether hypotheses A1 and A2 can be removed from Theorem~\ref{ncp:frcrsc} in the general case is an open problem. What follows is an example that conforms to Theorem~\ref{ncp:frcrsc}.

\begin{example}
	Consider the semidefinite constraint 
	\[
	G(x)\doteq \begin{bmatrix}
		-x_1+x_2^2+ x_2 & \;\;\;\; 2x_1 - 2x_2^2 + x_2  & \;\;\;\; x_1 - x_2^2 - x_2 \\
		2 x_1 - 2 x_2^2 + x_2 & \;\;\;\; -x_1 + x_2^2 + x_2& \;\;\;\; x_1 - x_2^2- x_2 \\
		x_1 - x_2^2 - x_2&\;\;\;\;  x_1 - x_2^2 - x_2  & \;\;\;\; 2 x_1 - 2 x_2^2 + x_2 
	\end{bmatrix}\in \S^3_+
	\]
	at the point $\xb=0\in\R^2$. Here, we have $\G=G$ and $\C=\S^3_{+}$. Moreover, for each $x\in \R^2$ and $d\in\R^2$ we have
	\[
D\G(x) d = (d_1-2x_2d_2)\begin{bmatrix}
-1&\;\;\;\;2&\;\;\;\;1\\
2&\;\;\;\;-1&\;\;\;\;1\\
1&\;\;\;\;1&\;\;\;\;2
\end{bmatrix}
+d_2\begin{bmatrix}
1&\;\;\;\;1&\;\;\;\;-1\\
1&\;\;\;\;1&\;\;\;\;-1\\
-1&\;\;\;\;-1&\;\;\;\;1
\end{bmatrix}.
	\]
	The eigenvalues of $ 	D\G(x) d $ are $3(d_1-2x_2d_2)$, $-3(d_1-2x_2d_2)$, and $3d_2$ which makes it clear that Robinson's condition is not valid at $\bar{x}=0$ and that 
	$$L_{\Omega}(x)=\{(d_1,d_2) \in  \mathbb{R}^2 \;\;| \;\; d_1 = 2x_2d_2, d_2 \geq 0 \} $$ 
	and
	\[
	\Gamma_{\C}(x) = D\G(x)[L_\Omega(x)] = \R^+u_1u_1^T,
	\]
	where $u_1=[1,1,-1]^T$. Notice that $\Gamma_{\C}(x)$ does not depend on $x$ and that it is a face of $\S^3_+$.  Hence it coincides with the minimal face that contains it, that is, we have $F_{J_-}=\Gamma_{\C}(x)$ for all $x$.  Also, it is simple to see that $H(x)=\R^2$, therefore closed. We conclude that hypotheses A1 and A2 of Theorem \ref{ncp:frcrsc} are satisfied. A simple computation gives that $F_{J_-}^\perp$ is the subspace generated by the matrices $u_2u_2^T$ and $u_3u_3^T$ where $u_2=[-1,1,0]^T$ and $u_3=[1,1,2]^T$, which allows us to compute $D\G(x)^*[F^\perp]=\R^2$ for all $x$. This proves that CRSC holds and we can use Theorem \ref{ncp:frcrsc} to conclude that there exists a neighborhood $\mathcal{U}$ of $\xb$ such that \[
	\G^{-1}(\C)\cap \mathcal{U}=\G^{-1}(F_{J_-})\cap \mathcal{U}.
	\]
That is, the constraint $\G(x)\in\S^3_+$ can be locally described as \[\G(x)\in F_{J_{-}}=\left\{
	\begin{bmatrix} 
		\beta &  \beta & -\beta\\ 
		\beta & \beta & - \beta \\
		-\beta & - \beta & \beta 	 	\end{bmatrix} \colon  \beta\geq 0	
	\right\}.\]
	This gives us that the $2\times 2$ principal block of $\G(x)$ must have equal entries, which implies $x_1=x_2^2$ while the remaining entries provide $x_2\geq0$. One can easily check that indeed the feasible set is globally equal to $\{ (x_1,x_2) \in \mathbb{R}^2 : x_1 = x_2^2, x_2 \geq 0 \}$.	
\end{example}

\section{Conclusions}

Constant rank constraint qualifications have become an important tool in the context of nonlinear programming due to its broad range of applications \cite{crcq,rcrcq,minch,aes2010,rcpld,cpg,gfrerer,abms}, being used in the study of tilt stability, strong second-order optimality conditions, global convergence of first- and second-order algorithms, derivative of the value function, among  others. Notice that these constraint qualification conditions allow some redundancy of the constraints; for instance, they hold for linear constraints without requiring removing linearly dependent constraints, which would be necessary under other conditions. This in particular gives someone building an optimization model more freedom, in the sense that redundant constraints are allowed to be part of the model. For these reasons, very recently, we started extending these conditions to the conic context by considering nonlinear second-order cone and semidefinite programming. To be precise, we started by showing that a previous attempt of defining such conditions for second-order cones was incorrect \cite{erratum} and a first correct approach was defined in \cite{naive}, by avoiding dealing with the conic constraints. A complete definition suitable for global convergence of algorithms appeared in \cite{seq-crcq-socp,seq-crcq-sdp}. Finally, a more geometric approach was presented in \cite{CRCQfaces} by exploiting the facial structure of semidefinite and second-order cones.

In this paper we extended the geometric approach to the more general case of reducible cones while also presenting two interesting applications, namely, we showed that a strong second-order optimality condition holds under the constant rank assumption and we showed that a facial reduction procedure is available under this assumption, rewriting the problem equivalently using only a smaller face of the cone in such a way that Robinson's condition is fulfilled. In the process of doing so, we noticed that the constant rank assumption could be enforced only along this particular face, which generalizes the results from \cite{cpg} to the conic context. This idea can also be carried out in order to obtain the strong second-order necessary optimality condition under a weaker condition, which is new even in nonlinear programming.

It is interesting to notice that the second-order necessary optimality condition we showed is stronger than the one obtained under Robinson's condition. In particular, it may hold even when the set of Lagrange multipliers is unbounded. We expect this second-order necessary optimality condition to be particularly relevant in proving second-order stationarity of primal-dual second-order algorithms for conic programming since it can be checked by means of a single dual approximation. This is not the case of the condition proven under Robinson's assumption, which depends on the full set of Lagrange multipliers and is thus not suitable for some dual algorithms. We were also able to prove that, under our facial constant rank condition, the numeric value of the second-order term is {\it independent} of the particular Lagrange multiplier used, which is not well adopted by the nonlinear programming community and is yet a topic to be fully exploited further. 

Finally, in this paper we rely on the reduction approach, simplifying the analysis to consider only points on the vertex of a reduced cone. However, we expect that our approach should be general enough in order to tackle more general conic constraints, which we will investigate in a future work. Also, in order to obtain our facial reduction result (Theorem \ref{ncp:frcrsc}), we assumed that the closedness of the set $H(\xb)$ and the minimality of the face $F=\Fmin(\Gamma_{\C}(\xb))$ are stable in a neighborhood of $\xb$ (assumptions A1 and A2) which could possibly also be improved in the future.\jo{\\

{\bf Declaration}

The authors declare that they have no conflict of interest.}

%
%

\bibliographystyle{plain}
\bibliography{biblio}

\end{document}